\documentclass[hidelinks,12pt,a4paper,reqno]{amsart}
\usepackage{graphicx}
\usepackage[all]{xy}
\usepackage{amssymb}
\usepackage{amsmath}
\usepackage[utf8]{inputenc}
\usepackage{amsthm}

\newtheorem{remark}{Remark}
\newtheorem{definition}{Definition}
\newtheorem{theorem}{Theorem}
\newtheorem{example}{Example}

\usepackage[backref=page]{hyperref}
\renewcommand*{\backref}[1]{}
\renewcommand*{\backrefalt}[4]{%
    \ifcase #1 (Not cited.)%
    \or        (Cited on page~#2.)%
    \else      (Cited on pages~#2.)%
    \fi}
\newcommand{\C}{\mathbf{C}}
\newcommand{\B}{\mathbf{B}}
\newcommand{\N}{\mathbb{N}}

\newcommand{\A}{\mathbf{A}}
\newcommand{\Grp}{\mathbf{Grp}}
\newcommand{\Mon}{\mathbf{Mon}}

\newcommand{\Set}{\mathbf{Set}}

\newcommand{\map}{\mathrm{map}}
\newcommand{\dom}{\mathrm{dom}}
\newcommand{\cod}{\mathrm{cod}}
\renewcommand{\span}{\mathrm{span}}

\newcommand{\Ker}{\mathrm{Ker}}

\begin{document}

\title[Semi-biproducts in monoids]{Semi-biproducts in monoids}


\author{N. Martins-Ferreira}
\address[Nelson Martins-Ferreira]{Instituto Politécnico de Leiria, Leiria, Portugal}
\thanks{ }
\email{martins.ferreira@ipleiria.pt}

\begin{abstract}
It is shown that the category of semi-biproducts in monoids is equivalent to a category of pseudo-actions. A semi-biproduct in monoids is at the same time a generalization of a semi-direct product in groups and a biproduct in commutative monoids. Every Schreier extension of monoids can be seen as an instance of a semi-biproduct; namely a semi-biproduct whose associated pseudo-action has a trivial correction system. A correction system is a new ingredient  that must be inserted in order to obtain a pseudo-action out of a pre-action and a factor system. In groups, every correction system is trivial. Hence, semi-biproducts there are the same as semi-direct products with a factor system, which are nothing but group extensions. An attempt to establish a general context in which to define semi-biproducts is made. As a result, a new structure of map-transformations is obtained from a category with a 2-cell structure. Examples and first basic properties are briefly explored.



\end{abstract}

\maketitle

\date{Received: date / Accepted: date}

\today

\section{Introduction}

The purpose of this paper is to introduce the notion of a semi-biproduct and to investigate it in the category of monoids.

From a categorical point of view, a biproduct is simultaneously a product and a co-product (see \cite{MacLane}, p.194). Furthermore, when there is a coincidence between products and co-products in a category, each hom-set carries the structure of an abelian group and parallel morphisms can be added. This means that a biproduct can be defined as a diagram of the form
\begin{equation}
\label{diag: biproduct in additive cats}
\xymatrix{X\ar@<-.5ex>[r]_{k} & A\ar@<-.5ex>@{->}[l]_{q}\ar@<.5ex>[r]^{p} & B \ar@{->}@<.5ex>[l]^{s}}
\end{equation}
such that the following conditions hold:
\begin{eqnarray}
ps=1_B\quad,\quad pk=0_{X,B}\label{eq: biproduct1}\\
qs=0_{B,X}\quad,\quad qk=1_{X}\label{eq: biproduct2}\\
kq+sp=1_A.\label{eq: biproduct3}
\end{eqnarray}

It is clear that this definition makes sense even when each hom-set is a commutative monoid. And that seems to be the most general case in which it makes sense. But we may ask the following question: how to generalize the definition so that it can be applied in the category of groups, where semi-direct products are expected to appear?

This work answers the question and  it goes beyond groups. It provides a generalization to the so called \emph{Schreier extensions} in monoids. 

A Schreier extension of monoids (see \cite{NMF et all}) can be seen as a sequence of monoid homomorphisms $$\xymatrix{X\ar[r]^{k}& A \ar[r]^{p} & B,}$$ in which $k=\ker(p)$, with the property that there exist two set-theoretical functions, $s\colon{B\to A}$ and $q\colon{A\to X}$, satisfying conditions $(\ref{eq: Sch-2})$--$(\ref{eq: Sch})$, below. When condition $(\ref{eq: Sch})$ is dropped, then the remaining ones can be reorganized so that they become precisely the conditions  $(\ref{eq: biproduct1})$--$(\ref{eq: biproduct3})$, except that now $q$ and $s$ are not homomorphisms but simply zero-preserving maps.

Thus, the notion of semi-biproduct in monoids arrives naturally as a diagram of the from $(\ref{diag: biproduct in additive cats})$ where $p$ and $k$ are monoid homomorphisms while $q$ and $s$ are zero-preserving maps, satisfying conditions $(\ref{eq: biproduct1})$--$(\ref{eq: biproduct3})$.

Let us see in more detail the long path that has given rise to this simple definition.

In  $\Grp$, the category of groups and group homomorphisms, every split epimorphism
\begin{equation*}
\label{diag: split epi}
\xymatrix{ A\ar@<.5ex>[r]^{p} & B \ar@{->}@<.5ex>[l]^{s}},\quad ps=1_B
\end{equation*}
induces an exact sequence
\begin{equation*}
\label{diag: exact seq}
\xymatrix{ X\ar[r]^{k} & A\ar@<.0ex>[r]^{p} & B, }
\end{equation*}
with  $X=\Ker(p)$. If fixing $X$ in additive notation (but not necessarily abelian) and $B$ in multiplicative notation, then,   up to isomorphism, the group $A$ is recovered as a semi-direct product of the form $X\rtimes_{\varphi}B$, with neutral element $(0,1)\in X\times B$, and group operation
\begin{equation}
(x,b)+(x',b')=(x+\varphi_b(x'),bb').
\end{equation}
Recall that (see e.g. \cite{MacLane2} and its references to previous work) $$\varphi_b\colon{X\to X}$$ is defined, for every $b\in B$, as the map $\varphi_b(x)=q(s(b)+k(x))$, with $q(a)=a-sp(a)\in X$, for every $a\in A$. Moreover, when $p$ is a surjective homomorphism, but not necessarily a split epimorphism, it is still possible to recover the group $A$ as a semi-direct product  with a factor system, $X\rtimes_{\varphi,\gamma}B$. A factor system, $\gamma\colon{B\times B\to X}$, is any map $\gamma(b,b')=q(s(b)+s(b'))=s(b)+s(b')-s(bb')$, for some chosen set theoretical section $s\colon{B\to A}$ of $p$, i.e., such that $ps=1_B$. In this case, the group operation in $X\rtimes_{\varphi,\gamma}B$ becomes
\begin{equation}
(x,b)+(x',b')=(x+\varphi_b(x')+\gamma(b,b'),bb').
\end{equation}

When $A$ is abelian, the action $\varphi_b(x)=s(b)+k(x)-s(b)\in X$ is trivial and so, the surjective homomorphisms with abelian domain are nothing but (symmetric) factor systems.

The case of monoids is different (see \cite{DB.NMF.AM.MS.13}). A surjective monoid homomorphism $p\colon{A\to B}$, in general, cannot be presented as a sequence
\begin{equation*}
\label{diag: exact seq mon}
\xymatrix{ X\ar[r]^(.34){\langle 1,0\rangle} & X\rtimes_{\varphi,\gamma}B \ar@<.0ex>[r]^(.64){\pi_B} & B }
\end{equation*}
with $X=\Ker(p)$ and such that the underlying set of $A$ is in bijection with the cartesian product $X\times B$. Thus, in monoids, the transport of the structure from  $A$ to $X\times B$ is not always possible. So, there is no hope of having an isomorphism $A\cong X\rtimes_{\varphi,\gamma}B$. As we will see, the best that we can hope for is to have $A$ (i.e. the underlying set of $A$) as a subset of the cartesian product. To do that we will need to introduce a new ingredient, a correction system, $\rho(x,b)=x^b$, so that we have 
$$A\cong (X\rtimes_{\varphi,\rho,\gamma} B)\subseteq X\times  B.$$  And yet, not every surjective monoid homomorphism  exhibits that property. Take, for example, the usual addition of natural numbers,
\begin{equation*}
\xymatrix{\N\times \N \ar[r]^{+} & \N,}
\end{equation*}
whose kernel is the trivial monoid.

We may ask: what kind of monoid homomorphisms has the same behaviour as in groups? The answer turns out to be the so called  Schreier extensions (see \cite{Ganci, NMF et all} and references therein, see also \cite{DB.NMF.AM.MS.13}). 


A Schreier extension $p\colon{A\to B}$ is nothing but a monoid homomorphism  with the property that there exists a set theoretical section $s\colon{B\to A}$ for which the function $X\times B\to A$, $(x,b)\mapsto k(x)+s(b)$ is a bijection of the underlying sets, with $k=\ker(p)$.

Equivalently, a Schreier extension can be seen as a monoid homomorphism $p\colon{A\to B}$, together with set theoretical maps $s\colon{B\to A}$ and $q\colon{A\to X}$, with $X=\Ker(p)$ and $k=\ker(p)\colon{X\to A}$, such that
\begin{eqnarray}
ps(b)=b &,& b\in B\label{eq: Sch-2}\\
kq(a)+sp(a)=a &,& a\in A\\
q(k(x)+s(b))=x &,& x\in X, b\in B. \label{eq: Sch}
\end{eqnarray}

The key feature of this work is the observation that,  at the expense of having to insert  a correction system, the condition $(\ref{eq: Sch})$ can be discarded. The correction system controls the lack of condition $(\ref{eq: Sch})$ and will be denoted by $x^b$. It is defined as $x^b=q(k(x)+s(b))$, for every $x\in X$ and $b\in B$. While dropping condition $(\ref{eq: Sch})$, we loose the bijection between the underlying set of $A$ and the cartesian product of sets $X\times B$. However, if adding two simple conditions, namely $q(k(x))=x$ and $q(s(b))=0$, then, $A$ is isomorphic to a subset of the cartesian product $X\times B$. The subset consists of those pairs in  $X\times B$ that are of the form $(x^b,b)$  for some $x\in X$ and $b\in B$ (see Theorem \ref{thm: semi-biproducts}). The monoid operation is
\begin{equation}
(x,b)+(x',b')=((x+b\cdot x'+(b\times b'))^{bb'},bb').
\end{equation}
The correction system $x^b$ has to satisfy some conditions (Definition \ref{def: pseudo-action}). In the case of groups it becomes trivial, that is $x^b=x$. To see it, simply take $b'=b^{-1}$ in equation $(\ref{eq: 53})$. The notion of a correction system is inspired by the work of Leech and Wells on extending groups by monoids (\cite{Leech,Wells}). The fact that the correction system $\rho(x,b)=x^b$ is invisible in groups, hides the true nature of pseudo-actions  as a combination of structures $\varphi$, $\rho$, $\gamma$ (Definition \ref{def: pseudo-action}) satisfying one single major condition (\ref{eq: factor}). Only the particular traces of (\ref{eq: factor}) are familiar in groups. This is better explained in Remark \ref{remark}. A Schreier extension in monoids  is precisely a semi-biproduct with a trivial correction system, which explains the similar behaviour between Schreier extensions in monoids and all extensions in groups. Examples of semi-biproducts in monoids which are not Schreier extensions, that is, with a non-trivial correction system, are presented (Section \ref{sec: eg}).

From a structural point of view, we are considering sequences of monoid homomorphisms
\begin{equation}
\xymatrix{X\ar[r]^{k} & A \ar[r]^{p} & B}
\end{equation} 
together with set theoretical maps,  preserving the neutral element, but not necessarily preserving the monoid operation,
\begin{equation}
\xymatrix{X & A \ar[l]_{q} & B \ar[l]_{s} }
\end{equation} 
and satisfying the conditions $(\ref{eq: biproduct1})$--$(\ref{eq: biproduct3})$. When the maps $q$ and $s$ are monoid homomorphisms then this is exactly the definition of a bi-product. It is then natural to call semi-biproduct when $q$ and $s$ are just zero-preserving maps.

As expected, a semi-biproduct becomes a product as soon as the zero-preserving maps $q$ and $s$ are both morphisms. In that case it is a coproduct as well. Particular instances when just one of $q$ or $s$ is a morphism are  worthwile studying, but we shall no longer expect a product or a co-product to appear. In general it does not make sense to speak about the semi-biproduct of two objects, unless some kind of pseudo-action  is specified. This is not so surprising since in groups, semi-biproducts are the same as semi-direct products with a factor system.  Connections with homology and cohomology are expected via several notions such as abstract kernel, obstruction, etc (see for example \cite{NMF et all} and references there), and it is clear that the notion of pseudo-action considered here (Definition \ref{def: pseudo-action}) can be used as a tool and clarify some classical  interpretations in group-cohomology.

This paper is organized as follows. First we concentrate our attention on the concrete case of monoids, introducing a new concept of pseudo-action and establishing a correspondence between pseudo-actions and semi-biproducts in monoids (Theorem \ref{thm: equivalence}). In Section \ref{sec: eg}  we illustrate the concept with Schreier extensions and other types of more general extensions.

At the end we address the natural question of where to define semi-biproducts. We arrive at an abstract structure of map-transformations. Its origin comes from a convenient structure of 2-cells in monoids. This appears to be a convenient setting to work in a wider context while keeping the intuition from monoids. The observation that semi-biproducts can be defined in such cases is made precise. We briefly discuss some properties of map-transformation structures, such as the property of some monomorphisms being recognizers (Theorem \ref{Th: recognizers}). A list of examples and general procedures on how to generate structures of map-transformations is provided.

\section{Monoids, semi-biproducts and pseudo actions}


In this section we define semi-biproducts in monoids, introduce a new notion of pseudo-action, and prove that pseudo-actions are equivalent to semi-biproducts. Several works have been done in this direction but the emphasis has always been on the side of (Schreier) extensions and their classification (e.g. \cite{ NMF14,Fleischer,Ganci,NMF et all}). Here we consider a semi-biproduct as a mathematical object rather than something which arises from an extension with appropriate choices for a section and a retraction. In a certain sense this work goes in the direction of  \cite{GranJanelidzeSobral}. There, the Schreier condition $(\ref{eq: Sch})$ is still being considered, but it shows that it is possible to work on the wider context of unitary magmas.

\subsection{Semi-biproducts in monoids} 

At the end of this paper we establish a general setting in which semi-biproducts can be defined. When interpreted in the case of monoids (with zero-preserving maps), it gives a diagram of the shape 
\begin{equation}
\label{diag: biproduct}
\xymatrix{X\ar@<-.5ex>[r]_{k} & A\ar@<-.5ex>@{->}[l]_{q}\ar@<.5ex>[r]^{p} & B \ar@{->}@<.5ex>[l]^{s}}
\end{equation}
such that $p$, $k$, are monoid homomorphisms, while  $q$ and $s$ are zero-preserving maps, verifying the conditions
\begin{eqnarray}
ps&=&1_B\\
qk&=&1_X\\
kq+sp&=&1_A\\
pk&=&0_{X,B}\\
qs&=&0_{B,X}.
\end{eqnarray}

 It is possible to characterize all semi-biproducts in monoids with fixed $X$ and $B$. We will work with  additive notation on $X$ and   multiplicative notation on $B$, but neither is assumed to be commutative.


\begin{theorem}\label{thm: semi-biproducts}
Let $(X,A,B,p,k,q,s)$ be a semi-biproduct in the category of monoids with zero-preserving maps. If we put for every $b,b'\in B$ and $x\in X$, 
\begin{eqnarray}
b\cdot x=q(s(b)+k(x))\\
x^b=q(k(x)+s(b))\\
b\times b'=q(s(b)+s(b'))
\end{eqnarray}
then, for every $a,a'\in A$
\begin{equation}\label{eq: a+a'}
a+a'=k(q(a)+p(a)\cdot q(a')+(p(a)\times p(a')))+sp(a+a')
\end{equation}
and 
\begin{equation}\label{eq: kuv+sv}
a+a'=k(u^v)+s(v)
\end{equation}
with $u=q(a)+p(a)\cdot q(a')+(p(a)\times p(a'))\in X$ and $v=p(a+a')\in B$.

Moreover, the map $\beta\colon{A\to X\times B}$ with $\beta(a)=(q(a),p(a))$ is always injective. An element $(y,b)\in X\times B$ is in the image of $\beta$ if and only if  $y=x^b$ for some $x\in X$.

This means that $A$ is in bijection with the set $\{(x^b,b)\mid (x,b)\in X\times B\}$ with the inverse map for $\beta$ being $\alpha(x^b,b)=k(x^b)+s(b)=k(x)+s(b)$. 

Furthermore, for every $a\in A$, $q(a)=q(a)^{p(a)}$.
\end{theorem}
\begin{proof}
We observe:
\begin{eqnarray*}
a+a'&=& kqa+(spa+kqa')+spa'  \qquad(kq+sp=1)\\
    &=& kqa+kq(spa+kqa')+spa+spa'  \qquad( ps=1,pk=0)\\
    &=&  kqa+kq(spa+kqa')+kq(spa+spa')+s(pa+pa')\\
       &=& k(qa+q(spa+kqa')+q(spa+spa'))+sp(a+a') \\
       &=& k(q(a)+p(a)\cdot q(a')+(p(a)\times p(a')))+sp(a+a')
\end{eqnarray*}
We observe further that for every $a\in A$, $x\in X$, $b\in B$, if $a=k(x)+s(b)$ then $a=k(x^b)+s(b)$, which proves condition $(\ref{eq: kuv+sv})$. This follows from the fact that for every $a\in A$, we have $q(a)=q(a)^{p(a)}$. Indeed, $q(a)=q(kq(a)+sp(a))$ follows directly from $a=kq(a)+sp(a)$.

If $\beta(a)=\beta(a')$ then $q(a)=q(a')$ and $p(a)=p(a')$, hence $a=a'$. This proves $\beta$ is injective. 

If $(x,b)\in X\times B$ is of the form $x=q(a)$, $b=p(a)$ for some $a\in A$, then $x=x^b$ because $q(a)=q(a)^{p(a)}$ as already shown. If $(y,b)\in X\times B$ is of the form $y=x^b$ for some $x\in X$, then there exists $a=k(x)+s(b)\in A$ such that $\beta(a)=(x,b)$.

\end{proof} 

We will now show that there is an equivalence of categories between the category of semi-biproducts in monoids and a suitable category of pseudo-actions in monoids.

\subsection{Pseudo-actions in monoids} Let $X$ and $B$ be two monoids. For convenience, let us again use additive notation for $X$ and multiplicative notation for $B$. Note that neither one is assumed to be commutative.

\begin{definition}\label{def: pseudo-action}
A pseudo-action of $B$ on $X$ consists in three different components:
\begin{enumerate}
\item a pre-action, that is, a map
\[\varphi\colon{B\to X^{X}}\] associating to every element $b\in B$ the map $\varphi_{b}\colon{X\to X}$ which is denoted as $\varphi_{b}(x)=b\cdot x$ and satisfies the conditions
\begin{eqnarray}
1\cdot x&=&x,\quad \forall x\in X\\
b\cdot 0&=&0, \quad \forall b\in B
\end{eqnarray}

\item a correction system, that is, a map
\begin{eqnarray}
\rho\colon{X\times B\to X};\quad 
(x,b)\mapsto x^b
\end{eqnarray}
such that for all $x\in X$ and $b\in B$
\begin{eqnarray}
x^1=x\\
0^b=0
\end{eqnarray}

\item a factor system, that is, a map
\begin{eqnarray}
\gamma\colon{B\times B\to X};\quad 
(b,b')\mapsto b\times b'
\end{eqnarray}
such that for every $b\in B$
\begin{eqnarray}
1\times b=0=b\times 1
\end{eqnarray}
\end{enumerate}
The three components are related via the condition $(\ref{eq: factor})$ which must hold for every $x,x',x''\in X$ and $b,b',b''\in B$.
\begin{eqnarray}
(x+b\cdot((x'+b'\cdot x'' + (b'\times b''))^{b'b''})+(b\times b'b''))^{bb'b''}=\nonumber\\
=((x+b\cdot x'+(b\times b'))^{bb'}+bb'\cdot x''+(bb'\times b''))^{bb'b''}\label{eq: factor}
\end{eqnarray}
\end{definition}

The correction system, $\rho$, is used to correct the fact that in general $b\cdot(x+y)$ is not  equal to $b\cdot x+b\cdot y$. Instead, we have the equalities
 \begin{eqnarray}
(b\cdot(x+y))^b=(b\cdot x+b\cdot y)^b\label{eq: correction1}\\
(x+y)^b=(x+y^b)^b \label{eq: correction2}\\
(x^{b}+b\cdot y)^{b}=(x+(b\cdot y)^b)^b \label{eq: correction3}
\end{eqnarray}
which are obtained (see items \ref{item4}, \ref{item5}, \ref{item6} below) as particular cases of $(\ref{eq: factor})$. This correction is invisible in groups (see item \ref{item2} below) which perhaps explains the reason why it has never appeared explicitly as a structure before. Nevertheless, it has  implicitly been used by Leech \cite{Leech} and Wells \cite{Wells} (see also \cite{Fleischer}).

\begin{remark}\label{remark}
In  $(\ref{eq: factor})$ we may observe the following particular cases of interest:
\begin{enumerate}
\item If taking $x,x',x''$ to be zero then we get
\begin{equation}\label{eq: factsystem}
(b\cdot (b'\times b'')^{bb'}+(b\times b'b''))^{bb'b''}=
((b\times b')^{bb'}+(bb'\times b''))^{bb'b''}
\end{equation}
which, if we ignore the correction system, it becomes the usual formula for a factor system in monoids (see e.g. \cite{Ganci} and references therein).

\item\label{item2} If taking $x=x''=0$, $b=1$ then we get (by putting again $x$ in the place of $x'$, $b$ in the place of $b'$ and $b'$ in the place of $b''$ for readability)
\begin{equation}\label{eq: 53}
(x^{b}+b\times b')^{bb'}=(x+b\times b')^{bb'}
\end{equation}
which explains how different the correction system is from being an action. In particular, if $X$ is right cancellable and $B$ is a  group then the correction system is always trivial. Indeed, if we take $b'=b^{-1}$ we obtain $x^b+(b\times b^{-1})=x + (b\times b^{-1})$ and hence, cancelling out $b\times b^{-1}$, we get $x^b=x$.

\item If taking $x=x'=0$ and $b''=1$ then we get
\begin{equation}\label{eq: factsystemconj}
(b\cdot(b'\cdot x'')^{b'}+(b\times b'))^{bb'}=((b\times b')^{bb'}+bb'\cdot x'')^{bb'}
\end{equation} 
which, in groups, becomes the familiar expression
\begin{equation}
b\cdot(b'\cdot x'')=(b\times b')+bb'\cdot x''-(b\times b')
\end{equation} 
stating that the factor system $b\times b'$ measures, by conjugation, the distance between a pre-action $\varphi$ and an ordinary action.

\item\label{item4} If taking $b=b'=1$ and $x''=0$ then we get
\begin{equation}
(x+x')^{b''}=(x+x'^{b''})^{b''}
\end{equation}
which is exactly the same as $(\ref{eq: correction2})$

\item\label{item5} If taking $b=b''=1$ and $x'=0$ then we get
\begin{equation}
(x^{b'}+b'\cdot x'')^{b'}=(x+(b'\cdot x'')^{b'})^{b'}
\end{equation}
which is exactly the same as $(\ref{eq: correction3})$.

\item\label{item6} If taking $b'=b''=1$ and $x=0$ then we get
\begin{equation}
(b\cdot (x'+x''))^b=((b\cdot x')^b+b\cdot x'')^b
\end{equation}
which, if combined with $(\ref{eq: correction2})$ and $(\ref{eq: correction3})$, gives $(\ref{eq: correction1})$.

\end{enumerate}
\end{remark}

\subsection{The equivalence}

It will be clear at the end of this subsection that there is an equivalence between pseudo-actions and semi-biproducts in monoids.

\begin{theorem}\label{thm: pseudo-actions}
Let $(X,A,B,p,k,q,s)$ be a semi-biproduct in the category of monoids with zero-preserving maps. The system with  three components
\begin{eqnarray}
b\cdot x=q(s(b)+k(x))\\
x^b=q(k(x)+s(b))\\
b\times b'=q(s(b)+s(b'))
\end{eqnarray}
is a pseudo action from $B$ into $X$.
\end{theorem}
\begin{proof}
The heart of the proof relies on the decomposition 
\[a+a'=k(q(a)+p(a)\cdot q(a')+(p(a)\times p(a')))+sp(a+a')
\] which holds for every $a,a'\in A$, as proved in Theorem \ref{thm: semi-biproducts}.



To obtain $(\ref{eq: factor})$, we consider $a+(a'+a'')=(a+a')+a''$ with $x=q(a), x'=q(a'), x''=q(a'')$ and $b=p(a), b'=p(a'), b''=p(a'')$. 
To prove $b\cdot 0=0$, $b\times 1=0=1\times b$ and $0^b=0$ we use the fact that $qs(b)=0$ for all $b\in B$ and the other conditions are easily verified.
\end{proof}


This means that if starting with a semi-biproduct $(X,A,B,p,k,q,s)$ we obtain a pseudo-action $(\varphi,\rho,\gamma)$ with $\varphi_{b}(x)=b\cdot x$, $\rho(x,b)=x^b$ and $\gamma(b,b')=b\times b'$ as defined in Theorem \ref{thm: pseudo-actions}. We only need to make a  straightforward calculation to see that given a pseudo-action we obtain a semi-biproduct in monoids. All the necessary ingredients have been given in Theorem \ref{thm: semi-biproducts}. Let us make it more precise. Given a pseudo-action $(\varphi,\rho,\gamma)$ of $B$ on $X$, we construct a synthetic (see \cite{Leech}) semi-biproduct in monoids
\begin{equation}
\xymatrix{X\ar@<-.5ex>[r]_{k} & A\ar@<-.5ex>@{->}[l]_{q}\ar@<.5ex>[r]^{p} & B \ar@{->}@<.5ex>[l]^{s}}
\end{equation}
with $A=\{(x^b,b)\mid x\in X, b\in B\}$ a monoid with neutral element $(0,1)\in X\times B$ and whose operation  is the restriction to the operation defined on $X\times B$ as
\begin{equation}
(x,b)+(x',b')=((x+b\cdot x'+(b\times b'))^{bb'},bb')
\end{equation}
 which is well defined and associative on the subset $A$. However, in general, it fails to be associative on the set $X\times B$. The homomorphisms $p$ and $k$, as well as the zero-preserving maps $q$ and $s$, are the expected ones.
 The monoid $A$ is in some sense a semi-direct product with a correction system and a factor system, so, it would be appropriate to write it as $X\rtimes_{\varphi,\rho,\gamma}B$.

\begin{theorem}\label{thm: equivalence}
There is an equivalence between the category of semi-biproducts in mo\-noids and the category of pseudo-actions.
\end{theorem}
\begin{proof}
Start with a semi-biproduct $(X,A,B,p,k,q,s)$ and let $(\varphi,\rho,\gamma)$ be its associated pseudo action. The synthetic semi-biproduct constructed with the pseudo-action is isomorphic to the original semi-biproduct as illustrated
\begin{equation}
\xymatrix{X\ar@<-.5ex>[r]_{k}\ar@{=}[d] & A\ar@{->}@<.5ex>[d]^{\beta}\ar@<-.5ex>@{-->}[l]_{q}\ar@<.5ex>[r]^{p} & B\ar@{=}[d] \ar@{->}@<.5ex>@{-->}[l]^{s}\\
X\ar@<-.5ex>[r]_(.4){\langle 1,0 \rangle} & X\rtimes_{\varphi,\rho,\gamma}B\ar@{->}@<.5ex>[u]^{\alpha} \ar@<-.5ex>@{-->}[l]_(.6){\pi_{1}}\ar@<.5ex>[r]^(.6){\pi_2} & B \ar@{->}@<.5ex>@{-->}[l]^(.4){{\langle 0 ,1 \rangle}}}
\end{equation}
with the morphisms  $\alpha$ and $\beta$ defined as in Theorem \ref{thm: semi-biproducts}.

If we start with a pseudo-action $(\varphi,\rho,\gamma)$, build its associated synthetic semi-biproduct, and extract the pseudo-action associated to it, then a new pseudo-action is derived, say $(\varphi',\rho',\gamma')$. The derived pseudo-action is obtained as $\varphi'_b(x)=(b\cdot x)^b$, $\rho'(x,b)$ is the same as $\rho(x,b)=x^b$ and $\gamma'(b,b')=(b\times b')^{bb'}$. Nevertheless these two pseudo-actions are isomorphic because $(x^b)^b=x^b$, which is a consequence of equation (\ref{eq: correction2}).
\end{proof}

\section{Examples of semi-biproducts in monoids}\label{sec: eg}

Every Schreier split extension (\cite{NMF14}) of a monoid $X$ by a monoid $B$  is a semi-biproduct in monoids. In this case $b\times b'=0$ and $x^b=x$ for all $x\in X$, $b,b'\in B$. Every Schreier extension of monoids is a semi-biproduct in which $x^b=x$. In fact, a semi-biproduct of monoids $(X,A,B,p,k,q,s)$ is a Schreier extension if and only if $x^b=x$ for all $x\in X$ and $b\in B$.

Here is a general procedure for the construction of semi-biproducts which are not necessarily Schreier extensions. 

Let $X$ (written additively) and $B$ (written multiplicatively) be two monoids and let us suppose the existence of a subset $R\subseteq X\times B$, considered as a binary relation, so that we write $xRb$ instead of $(x,b)\in R$, together with two maps $u\colon{B\to X}$ and $q\colon{R\to X}$ satisfying the following conditions:
\begin{eqnarray}
\text{ for all $x\in X$, $xR1$ and $q(xR1)=x$}\\
\text{ for all $b\in B$, $u(b)Rb$ and $q(u(b)Rb)=0$}\\
\text{ if $xRb$, $yRb$ and $q(xRb)=q(yRb)$ then $x=y$.}\label{item:q}
\end{eqnarray}
For every monoid structure on the set $R$, for which $0R1$ is the neutral element and the projection map $xRb\mapsto b$ is a homomorphism, we put:
\begin{eqnarray*}
x\oplus x'&=&q(xR1+x'R1)\\
b\times b'&=&q(u(b)Rb+u(b')Rb')\\
b\cdot x&=&q(u(b)Rb+xR1)\\
x^b&=&q(xR1+u(b)Rb).
\end{eqnarray*}

If $x\oplus x'=x+x'$ for all $x,x'\in X$ (where $+$ is the monoid operation on $X$) and if $q(xRb)^b=q(xRb)$, whenever $xRb$, then we have a semi-biproduct
\begin{equation}
\xymatrix{X\ar@<-.5ex>[r]_{k} & R\ar@<-.5ex>@{-->}[l]_{q}\ar@<.5ex>[r]^{p} & B \ar@{-->}@<.5ex>[l]^{s}}
\end{equation}
with $p(xRb)=b$, $k(x)=(xR1)$, $s(b)=(u(b)Rb)$.

 The monoid operation on $R$ is given by the formula
\begin{equation}
(xRb)+(x'Rb')=(x+b\cdot x'+ b\times b')^{bb'}Rbb'.
\end{equation}
Indeed, $k$ is a homomorphism because $\oplus$ is the monoid operation defined on $X$, $p$ is a homomorphism by assumption, $pk=0$, $ps=1$, $qk=1$, $qs=0$ by construction  and we have, for every $xRb$,
\begin{equation}\label{eq:xRb}
xRb=q(xRb)R1+u(b)Rb
\end{equation}
which means that $kq+sp=1_R$. It is not difficult to see that $(\ref{eq:xRb})$ follows from the fact that $q(xRb)^b=q(xRb)$ together with the fact that the map $\langle q,p\rangle\colon{R\to X\times B}$ is injective (condition $(\ref{item:q})$ above). 


In particular, when we force $q(xRb)=x$ and $u(b)=0$ then we are reduced to the analysis of monoid structures on $R$. In that case, the condition $q(xRb)^b=q(xRb)$ becomes $x^b=x$ whenever $xRb$, and if we are looking for semi-biproducts which are not Schreier extensions then we have to consider relations which are not in bijection with the set $X\times B$. Let us work a concrete example.

Take $X=\{0,s\}$ with $s+s=s$ and $B=\{1,t\}$ with $t^2=t$. Clearly, $R=\{(0,1),(s,1),(0,t)\}$ is the only possible proper subset of $X\times B$ with $xR1$ and $0Rb$ for every $x\in X$ and $b\in B$. There are two possible solutions to turn $R$ into a monoid with $0R1$ as neutral element and so that the map $p(xRb)=b$ is a homomorphism.

One possibility is to map $0R1\mapsto 1$, $sR1\mapsto -1$, $0Rt\mapsto 0$ and take the usual multiplication on the set $\{-1,0,1\}$. However, in this case we find that $s\oplus s=0$ while $s+s=s$. So, it does not give a semi-biproduct.

Another possibility is to consider the chain semilattice structure on $R$ as follows
\begin{equation*}
\begin{tabular}{c|ccc}
+ & $0R1$ & $sR1$ & $0Rt$\\
\hline
$0R1$ & $0R1$ & $sR1$ & $0Rt$\\
$sR1$ & $sR1$ & $sR1$ & $0Rt$\\
$0Rt$ & $0Rt$ & $0Rt$ & $0Rt$
\end{tabular}
\end{equation*}
In this case we have
\begin{equation*}
\begin{tabular}{c|c|c|c|c|c|c|c}
$x$ & $b$ & $x'$ & $b'$ & $x\oplus x'$ & $b\times b'$ & $b\cdot x$ & $x^b$\\
\hline
$0$ & $1$ & $0$ & $1$ & $0$ & $0$ & $0$ & $0$ \\
$0$ & $t$ & $s$ & $1$ & $s$ & $0$ & $0$ & $0$ \\
$s$ & $1$ & $0$ & $t$ & $s$ & $0$ & $s$ & $s$ \\
$s$ & $t$ & $s$ & $t$ & $s$ & $0$ & $0$ & $0$ 
\end{tabular}
\end{equation*}
and since $x\oplus x'=x+x'$ and $x^b=x$ whenever $xRb$ (note that $s^t=0$ but $(s,t)\notin R$) we obtain a semi-biproduct of monoids 
\begin{equation}
\xymatrix{\{0,s\}\ar@<-.5ex>[r]_{\iota_1} & R\ar@<-.5ex>@{-->}[l]_{\pi_1}\ar@<.5ex>[r]^{\pi_2} & \{1,t\} \ar@{->}@<.5ex>[l]^{\iota_2}}
\end{equation}
with $\pi_2(xRb)=b$, $\iota_2(b)=(0Rb)$, $\pi_1(xRb)=x$, $\iota_1(x)=(xR1)$. It cannot be a Schreier extension because $R$ has three elements while $X\times B$ has four elements.

Let us now see a simple example which illustrates the case when we take $R$ to be in bijection with $X\times B$. Put $X=B=\mathbb{N}$ the additive monoid of natural numbers and consider the order relation $R=\{(x,b)\in \mathbb{N}^2\mid x\geq b\}$ together with the two maps $q(x\geq b)=x-b$ and $u(b)=b$. The usual component-wise addition on $R$ verifies all the conditions specified by the  construction scheme outlined above and hence it gives rise to a semi-biproduct. Namely, $(\mathbb{N},R,\mathbb{N},p,k,q,s)$ with $p(x\geq b)=b$, $k(x)=(x\geq 0)$, $q(x\geq b)=x-b$ and $s(b)=(b\geq b)$. In this case $q$ and $s$ are both homomorphisms and hence $x^b=x$ for every $(x,y)\in \mathbb{N}\times \mathbb{N}$ so that we obtain an isomorphism
\[
\xymatrix{\mathbb{N}\ar@{=}[d]_{}\ar@<-.5ex>[r]_{k} & R\ar@<-.5ex>[d]_{\beta}\ar@<-.5ex>@{->}[l]_{q}\ar@<.5ex>[r]^{p} & \mathbb{N}\ar@{=}[d]^{} \ar@{->}@<.5ex>[l]^{s}\\
\mathbb{N}\ar@<-.5ex>[r]_{\langle 1,0\rangle} & \mathbb{N}\times \mathbb{N}\ar@<-.5ex>[u]_{\alpha}\ar@<-.5ex>@{->}[l]_{\pi_1}\ar@<.5ex>[r]^{\pi_2} & \mathbb{N} \ar@{->}@<.5ex>[l]^{\langle 0,1\rangle}}
\]
with $\beta(x\geq b)=(x-b,b)$ and $\alpha(x,b)=(x+b\geq b)$. See Theorem \ref{thm: equivalence}.

\section{What is the right context in which to define semi-biproducts?}
   
The question of choosing an appropriate setting in which to define semi-biproducts is a natural one. We will merely give some directions that we expect to follow in future research. It all starts with the observation that zero-preserving maps, an essential part on the definition of semi-biproduct (mainly because they are closed under component-wise addition), can be seen as a particular instance of a 2-cell structure in the category of monoids. The following subsections are an attempt to explain better what we have in mind.

\subsection{The category of monoids with a 2-cell structure}

It is well known that a monoid can be seen as a one object category. This means that two parallel monoid homomorphisms, $f,g\colon{A\to B}$, can be seen as two functors. Thus, a natural transformation, $\tau\colon{f\Longrightarrow g}$ becomes nothing but an element $t\in B$ for which $t+f(x)=g(x)+t$ for all $x\in A$. In this case we write $\tau=(f,t,g)$. Pushing it a little bit further we may say that a \emph{transformation} from $f$ to $g$ is a zero-preserving map $t\colon{A\to B}$ such that $t(x)+f(x)=g(x)+t(x)$, for all $x\in A$. This gives rise to an example of a 2-cell structure in the sense of \cite{NMF.15}. See also \cite{Brown}, where they are called \emph{wiskered categories}. A wiskered category, or a category with a 2-cell structure, is not a 2-category because the middle interchange law is not present.

 A 2-cell structure over a category $\A$ is a system $(H,\dom,\cod,0,+)$ in which $H\colon{\A^{\text{op}}\times \A\to \Set}$ is a bifunctor, $\dom,\cod\colon{H\to\hom_{\A}}$, $0\colon{\hom_{\A}\to H}$ and $$+\colon{H\times_{\hom} H\to H}$$ are natural transformations satisfying certain conditions (the details can be found in \cite{NMF.15} but will not be needed here). We only need to observe that if defining, for every two monoids $A$ and $B$, the set $H(A,B)$ as consisting in all the triples of the form $\tau=(f,t,g)$, with $f,g\colon{A\to B}$ monoid homomorphisms, and $t\colon{A\to B}$ a zero-preserving map with the property  $t(x)+f(x)=g(x)+t(x)$, for all $x\in A$, then it gives rise to a 2-cell structure on the category of monoids. Indeed, $\dom(\tau)=f$, $\cod(\tau)=g$, $0(f)=(f,0_{A,B},f)$ and if  $\tau'=(g,t',h)$ then $\tau'+\tau=(f,t'+t,h)$.
 
We point out that this is not a natural 2-cell structure  and hence it does not give rise to a horizontal composition as would be expected in a 2-category.
  Nevertheless, we do have a reasonable formula for the horizontal composition of 2-cells, namely
  \begin{equation}\label{eq: horizontal composition}
(f,t,g)\circ(f',t',g')=(ff',tt',gg')
\end{equation}
for  homomorphisms $f,f',g,g'$ and zero-preserving maps $t$ and $t'$ with appropriate domains and codomains. In a small parenthesis we may add that for the middle interchange law to hold, it would be needed that $gt'+tf'$ is equal to $tg'+ft'$, which is not the case in general.
 
 Note that in $\tau=(f,t,g)$, when $f$ and $g$ are the zero homomorphisms, then $t$ becomes a zero-preserving map with no other conditions.
 
 We conclude this subsection by remarking that if we denote by $\map(A,B)$ the collection of all zero-preserving maps from a monoid $A$ into a monoid $B$, then, it is derived from the bifunctor $H$, as defined above, with\[\map(A,B)=\{\tau\in H(A,B)\mid \dom(\tau)=0_{A,B}=\cod(\tau)\}.\]
 
 This suggests to attempt for an abstraction of $\map(A,B)$. It is a bifunctor into the category of pointed sets (as soon as the category $\A$ is pointed) and it factors through the category of monoids. Moreover, there is a natural inclusion $\varepsilon\colon{\hom_{\A}\to\map}$ which, in the current example of monoids, is defined as $\varepsilon(f)=(0_{A,B},f,0_{A,B})$ for every homomorphism $f\colon{A\to B}$. Its compatibility with $(\ref{eq: horizontal composition})$,  when restricted from $H$ to $\map$, is expressed as follows:
 \begin{eqnarray}
 (0,t,0)\circ \varepsilon(h)=(0,th,0)=(0,t,0)h,\\
 \varepsilon(g)\circ (0,t,0)=(0,gt,0)=g(0,t,0).
 \end{eqnarray}
    The formulas on the right are the wiskering compositions between morphisms and 2-cells. Compare these with $(\ref{eq: xf})$ and $(\ref{eq: gx})$ below.
   
\subsection{A (pointed) category with a map-transformation structure}\label{subsec: basepoint structure}

In order to speak of semi-biproducts we need to introduce a setting where addition of morphisms is possible. As follows from the case in monoids, the (componentwise) addition of two parallel morphisms fails, in general, to be a morphism and it is only a zero-preserving function. This fact, together with the explanation provided in the previous subsection, suggests to introduce the notion of a (pointed) category with an abstract  structure of zero-preserving maps, that we decided to call map-transformations, for convenience of notation.

Let  $\A$ be a pointed category and let $\hom_{\A}$ denote its hom-functor. A  \emph{map-transformation structure} on $\A$ consists in a bifunctor 
 \begin{equation}
 \map\colon{\A^{\mathrm{op}}\times \A \to \Set_{*}}
 \end{equation}
that factors through the category of monoids
\begin{equation}\label{diag:Map}
\xymatrix{\A^{\mathrm{op}}\times \A \ar[r]^{\mathrm{Map}} \ar[dr]_{\map}& \Mon\ar[d]_{F}\\&\Set_{*}}
\end{equation}
together with a natural inclusion (of pointed sets)
\begin{equation}
\varepsilon\colon{\hom_{\A} \to  \map}
\end{equation}
and, for every three objects $A,B,C$ in $\A$, an associative composition of \emph{map-transformations}
\begin{equation}
\mu_{A,B,C}\colon{\map(B,C)\times\map(A,B)\to\map(A,C)}
\end{equation}
such that, for every $x\in \map(A,B)$, and for every two morphisms $f\colon{A'\to A}$, $g\colon{B\to B'}$, 
\begin{equation}
\mu(\varepsilon(f),x)=\map(f,1)(x)
\end{equation}
and 
\begin{equation}
\mu(x,\varepsilon(g))=\map(1,g)(x).
\end{equation}

If we put $gxf=\map(f,g)(x)$ and use $gx$ and $xf$ instead of $gx1$ and $1xf$, then, the conditions on a map-transformation structure translate as
\begin{eqnarray}
g0f=0\\
g(x+y)f=gxf+gyf\label{eq21}\\
1x1=x\\
g'(gxf)f'=(g'g)x(ff')\\
\varepsilon(0)=0\\
g\varepsilon(u)f=\varepsilon(guf)\\
\mu(x,\varepsilon(f))=xf\label{eq: xf}\\
\mu(\varepsilon(g),x)=gx\label{eq: gx}\\
\mu(x,\mu(x',x''))=\mu(\mu(x,x'),x'')
\end{eqnarray}
for every $x,y\in\map(A,B)$, and $u\colon{A\to B}$, $f\colon{A'\to A}$, $f'\colon{A''\to A'}$, $g\colon{B\to B'}$, $g'\colon{B'\to B''}$, and $x'\in\map(A',A)$, $x''\in\map(A'',A')$.

The elements in $\map(A,B)$ are called map-transformations from $A$ to $B$ and are represented with double arrows. This means that an element $x\in\map(A,B)$ is  displayed as
\[\xymatrix{A\ar@{=>}[r]^x & B}\] and the element $gxf$ is obtained as
\begin{equation}\label{diag: gxf}
\xymatrix{A\ar@{=>}[r]^{x} & B \ar[d]^{g}
\\ A'\ar[u]^{f} \ar@{=>}[r]_{gxf} & B'.}
\end{equation}

This definition and its interpretation, namely viewing an element $x\in\map(A,B)$ as a super-morphism from $A$ to $B$ is motivated with the striking analogy between map-transformations and 2-cells in a category with a 2-cell structure (or a sesquicategory) as studied in \cite{NMF.15}. Indeed, the list of conditions above is very similar to a list in \cite{NMF.15}, characterizing 2-cell structures in terms of a vertical composition (written additively) and left and right actions (whiskering composition).

The use of the word \emph{transformation} seems to be appropriate in the sense that it captures the idea of something like a morphism, but different. In a category with a map-transformation structure every morphism, $f$, is also a (map-)transformation, $\varepsilon(f)$, but not every transformation is a morphism. Let us see some examples.

\subsection{Examples of categories with map-transformations}

A simple case is  the category of commutative monoids where we can take the bifunctor $\map$ to be precisely the hom-functor. In this case there is no distinction between map-transformations and morphisms. We emphasize the fact that the name map-transformation is being used as a designation for the elements in a given bifunctor $\map(A,B)$  which is part of a map-transformation structure, following the same analogy in which the word \emph{morphism} is used to designate an element in $\hom(A,B)$. Let us see other examples where a map-transformation is not necessarily a morphism.

\begin{example}
Let $\A$ be a pointed category. We can always turn $\A$ into a category with a map-transformation structure. Indeed, the adjoint to the forgetful functor  from monoids to pointed sets can be used to turn each $\hom_{\A}(A,B)$, a pointed set, into a monoid. The  process is standard: generate a free monoid and identify the zero morphism with the neutral element. Let us denote by $\map(A,B)$ the result of that process. A typical element in $\map(A,B)$ is thus a formal sum $h_1+\cdots +h_n$ of homomorphisms $h_i\colon{A\to B}$ in $\A$. The inclusion $\varepsilon$ is clear and the composition $\mu$ is defined as follows. If $h=h_1+\cdots +h_n$ is an element in $\map(B,C)$ and $h'=h'_1+\cdots +h'_n$ is an element in $\map(A,B)$, then $\mu(h,h')\in \map(A,C)$ is the formal expression
$$h_1h'_1+\cdots +h_1h'_n+\cdots h_nh'_1+\cdots +h_nh'_n.$$
Observe that this expression is formed using distributivity on the right, but not on the left. This is in accordance with the fact that condition $(\ref{eq21})$ (in the case of monoids) holds for $f$ and $g$ homomorphisms, whereas if they are merely functions then we have $(x+y)f=xf+yf$ but $g(x+y)$ is not the same as $gx+gy$. 
\end{example}

This example is somehow trivial, but if, for  every object $A$ in $\A$, we choose a set of formal expressions to be identified with the identity morphism, $1_A$, then some interesting situations are expected in the study of semi-biproducts. Let us now see another sort of examples.

\begin{example} Let $\A$ be a pointed category with, say finite limits, and co-products. We can always define a map-transformation structure as follows. For each pair of objects $(A,B)$ let $\map(A,B)$ be $\span(A,B)_{/ \sim}$, that is, the collection of all spans (identifying equivalent ones) in $\A$ form $A$ to $B$. A span from $A$ to $B$ is an ordered pair of morphisms $(u,v)$  as illustrated
\[\xymatrix{A& S\ar[l]_{u}\ar[r]^{v} & B.}\]
Two spans are said to be equivalent is there is a zig-zag of span morphisms connecting them (see e.g. \cite{MacLane}).

The inclusion of morphisms into spans is done by the formula $\varepsilon(h)=(1_A,h)$, for every morphism $h\colon{A\to B}$ in $\A$. The usual composition of spans can be used to define the map $\mu((u,v),(t,s))=(u\bar{t},s\bar{v})$, as illustrated in the diagram below (in which $\bar{t}$ and $\bar{v}$ are obtained by pullback)
\[\xymatrix{&&\cdot \ar[ld]_{\bar{t}}\ar[rd]^{\bar{v}}\\& \cdot \ar[dl]_{u}\ar[dr]^{v} & & \cdot \ar[ld]_{t}\ar[rd]^{s} \\
\cdot & & \cdot & & \cdot}\]  
This formula, in particular, can be applied to the left and right composition of morphisms with spans, as to fulfil equations $(\ref{eq: xf})$ and $(\ref{eq: gx})$. It remains to define a monoid structure on each $\map(A,B)$. The neutral element is necessarily the insertion of the zero morphism, whereas addition is obtained as follows. Let $(u,v)$ and $(u',v')$ be two spans from $A$ to $B$, then we define $(u,v)+(u',v')=([u,u'],[v,v'])$ obtained via co-product as illustrated
\[\xymatrix{A& S+S'\ar[l]_(.6){[u,u']}\ar[r]^(.6){[v,v']} & B.}\]
Up to equivalence, this operation is associative, and the span $(1_A,0_{A,B})$ is the neutral element in $\map(A,B)$.
\end{example}

 Other technical details are omitted here since we will not, at the moment, analyse this example. Moreover, it is somehow trivial. Nevertheless, if we choose a good class of morphisms in $\A$, say $\mathcal{E}$, containing all isomorphisms, which is closed under composition, stable under pullbacks, and if $u,u'\in \mathcal{E}$ are two morphisms with the same codomain, then the induced morphism from the coproduct of their domains, $[u,u']$, into the common codomain, is still in $\mathcal{E}$. Then, we can  restrict the previous example and consider only those spans $(u,v)$ for which $u\in \mathcal{E}$. This will certainly create interesting examples. Take for instance the class of regular epimorphisms in a semi-abelian category.

The previous example is also related to the notion of imaginary morphisms in the following sense. When considering a good class of morphisms $\mathcal{E}$ it may happen that there exists a pointed endo-functor of $\A$, say $(T,\eta)$ for which $\map(A,B)\cong\hom_A(T(A),B)$ so that each class of spans in $\span(A,B)_{/\sim}$ with the left leg in $\mathcal{E}$ is represented by a span of the form
\[\xymatrix{A & T(A) \ar[l]_{\eta_A} \ar[r]^{f} & B.}\] 
This means that we can \emph{imagine}  a map-transformation from $A$ to $B$  as a morphism from $T(A)$ to $B$ (see for example \cite{MontoliRodeloLinden} and references therein). These ideas will be developed in future work.

Let us now see a procedure to construct more concrete examples, given some concrete situations.

\begin{example}
Let $\A$ be a pointed category and $I\colon{\A\to\Mon}$ a faithful functor. Let $F$ denote the forgetful functor from monoids into pointed sets. Then there is a bifunctor
\[\map\colon{\A^{op}\times \A \to \Mon} \]
with $\map(A,B)$ the collection of all pointed maps from the pointed set $FI(A)$ to the pointed set $FI(B)$. The zero  map $0_{A,B}\colon{FI(A)\to FI(B)}$ serves as  neutral element and the usual component-wise addition of pointed maps, that is, $(f+g)(a)=f(a)+g(a)\in FI(B)$, for every $a\in FI(a)$ and $f,g\in \map(A,B)$, makes $\map(A,B)$ a monoid. It is clear that the element $f(a)+g(a)$ comes from the monoid $I(B)$.
 
 The natural inclusion $\varepsilon$ as well as the associative composition $\mu$ are not difficult to obtain in order to make $\A$ a category with map-transformations. Indeed, if $h\colon{A\to B}$ is a morphism in $\A$ then $\varepsilon(h)$ is simply the map $FI(h)$, while $\mu(f,g)=fg$ is nothing but the usual composition of set-theoretical pointed maps.

\end{example}

The previous example (as well as the following one) can be generalized to the case where $\Mon$ is replaced with $\Mon(\C)$, i.e. the category of internal monoids in a category $\C$ with finite limits. Furthermore, it is clear that $\map(A,B)$ does not need to be equal to the collection of all pointed maps from $FI(A)$ to $FI(B)$. Let us make this assertion more precise.

\begin{example}
For every commutative square of categories and functors
\[\xymatrix{&\A\ar[ld]_{V}\ar[rd]^{I}\\\B\ar[rd]_{U}&& \Mon \ar[ld]^{F}\\&\Set_{*}}\]
in which $\A$ and $\B$ are pointed categories and all the functors are faithful (recall that $F$ is the  forgetful functor from monoids to pointed sets) we can always find a family of submonoids $$j_{A,B}\colon{\map(A,B)\to \hom_{Set_{*}}(FI(A),FI(B))}$$ satisfying the following two conditions:
\begin{enumerate}
\item  the inclusion of $\hom_{\B}(V(A),V(B))$, for every pair or objects $A,B$, into $\hom_{\Set_{*}}(FI(A),FI(B))$, factors through $j_{A,B}$, say, as
\[\xymatrix{\hom_{\B}(V(A),V(B))\ar[rd]\ar@{-->}[r]^{i_{A,B}} & \map(A,B)\ar[d]_{j_{A,B}}\\ &\hom_{\Set_{*}}(FI(A),FI(B))}\]
\item for every three objects $A,B,C$, the restriction to the composition mapping of pointed maps, factors as follows
\[\xymatrix{\map(B,C)\times \map(A,B)\ar[dr]\ar@{-->}[r]^(.6){\mu_{A,B,C}} & \map(A,C)\ar[d]_{j_{A,B}}\\ &\hom_{\Set_{*}}(FI(A),FI(C))}\]
\end{enumerate}
In this way we get a bifunctor $\map$ which is defined on morphisms, say $f\colon{A'\to A}$, and  $g\colon{B\to B'}$,
\[\map(f,g)\colon{\map(A,B)\to \map(A',B')}\]
as $map(f,g)(u)=\mu(i(V(f)),\mu(u,i(V(g))))$. It is a monoid homomorphism because $I(f)$ is a monoid homomorphism and $UV(f)=FI(f)$. The inclusion of morphisms into map-transformations is done with the mapping $\varepsilon(h)=i(V(h))$.

\end{example}

In this paper we have considered the case when $\A=\B=\Mon$, $I=V$ is the identity functor, and $U=F$. It is, nevertheless, clear that several interesting other examples can be considered as well:
\begin{enumerate}
\item Take $\A$ to be the category of topological groups with $I$ its forgetful functor into monoids and $V$ its forgetful functor into topological spaces;
\item Take $\A$ to be the category of topological monoids with $I$ its forgetful functor into monoids and $V$ its forgetful functor into topological spaces;
\item Take $\A$ to be the category of (pre)ordered groups with $I$ its forgetful functor into monoids and $V$ its forgetful functor into (pre)ordered sets;
\end{enumerate}

Indeed, when $\A$ is the category of topological groups then we have two possibilities for $\map(A,B)$ (see for example \cite{Ganci}). It can be the set of pointed maps from $FI(A)$ to $FI(B)$, or the set of continuous  pointed maps from $FI(A)$ to $FI(B)$. In the same way, the category of (pre)ordered groups (with it's forgetfull functor into monoids) can be enriched with a map-transformation structure in two different ways. We can take $\map(A,B)$ to be the set of pointed maps from $FI(A)$ to $FI(B)$ but we can also take it to be the set of monotone pointed maps from $FI(A)$ to $FI(B)$. This idea is implicitly used in \cite{Preord} where a systematic study of semi-direct products is conducted in the category of preordered groups. Other possibilities are also interesting to be considered, such is the case of those pointed maps $\tau$ which are not homomorphisms but satisfy the condition $\tau(x+y)\leq \tau(x)+\tau(y)$.

In a certain sense, all examples come from a structure of 2-cells in a pointed category with a horizontal composition. All the details that are omitted in the example below can be found in \cite{NMF.15}, see also \cite{Brown}.

\begin{example}
Let $\A$ be a pointed category equipped with a 2-cell structure $(H,\dom,\cod,0,+)$. Let us suppose it comes equipped with a horizontal composition law for 2-cells, $\mu$, not necessarily obtained from the middle interchange law. Define, $$\map(A,B)=\{x\in H(A,B)\mid \dom(x)=0=\cod(x)\}.$$ Now, for every natural inclusion of pointed sets, $\varepsilon\colon{\hom\to\map}$, which satisfies conditions $(\ref{eq: xf})$ and $(\ref{eq: gx})$ we obtain a map-transformation structure $(\map,\mu,\varepsilon)$ on $\A$.
\end{example}

\section{Semi-biproducts}

In this section we give a definition of semi-biproduct in a category with an abstract map-transformation structure and derive some of its immediate properties. 

\subsection{Semi-biproducts in categories with map-transformations}
A semi-biproduct, as a generalization of biproduct, makes sence in a category with map-transformations. Indeed, it is obtained from a semi-biproduct in monoids, such as the one displayed in diagram $(\ref{diag: biproduct})$, by declaring that the zero-preserving maps $q$ and $s$ are map-transformations (and reformulating the conditions accordingly so that they still make sense).

\begin{definition}\label{def: semi-biproduct}
 Let $\A=(\A,\map,\varepsilon,\mu)$ be a category with a structure of map-transformations. A semi-biproduct in $\A$ is a diagram of the form
\begin{equation}
\label{diag:semi-biproduct}
\xymatrix{X\ar@<-.5ex>[r]_{k} & A\ar@<-.5ex>@{=>}[l]_{q}\ar@<.5ex>[r]^{p} & B \ar@{=>}@<.5ex>[l]^{s}}
\end{equation}
in which $X$, $A$, $B$ are objects in $\A$, $p$, $k$ are morphisms, $q$, $s$ are map-transformations, and the following conditions hold:
\begin{eqnarray}
ps&=&\varepsilon(1_B)\\
qk&=&\varepsilon(1_X)\\
kq+sp&=&\varepsilon(1_A)\\
pk&=&0_{X,B}\\
\mu(q,s)&=&\varepsilon(0_{B,X}).
\end{eqnarray}
\end{definition}
Recall that $ps=\map(1_B,p)(s)\in\map(B,B)$, $qk=\map(k,1_X)(q)\in\map(X,X)$. Similarly, $kq$ and $sp$ are elements in $\map(A,A)$ and $kq+sp\in\map(A,A)$ is their sum using the monoid operation defined in the monoid $\mathrm{Map}(A,A)$, see diagram $(\ref{diag:Map})$.

We will see some immediate properties of this definition, further studies are postponed to future work.

\subsection{Immediate properties of semi-biproducts} Let us suppose that $(X,A,B,p,k,q,s)$ is a  semi-biproduct in a category with a structure of map-transformations, such as the one displayed in diagram (\ref{diag:semi-biproduct}). It follows that if  $k$ has a certain desirable property, then it is a kernel for $p$ and dually $p$ is a cokernel for $k$. The desirable property is the ability to \emph{recognize} when a map-transformation, say $u$, is in the image of $\varepsilon$ or not, that is, if it is of the form $u=\varepsilon(f)$ for some morphism $f$ or not. A morphism with this property is said to be a recognizer.

\begin{definition}
A monomorphism $k\colon{X\to A}$ is said to be  a \emph{recognizer} when for every object $Y$ and map-transformation $u\in\map(Y,X)$, if there exists $f\colon{Y\to A}$ such that $ku=\varepsilon(f)$ then $u=\varepsilon(u')$ for some $u'\colon{Y\to X}$.
\end{definition}

In other words, $k$ recognizes morphisms if, whenever the composition of $k$ with a map-transformation $u$ results in a morphism $ku$, then, $u$ is a morphism itself.

Dually, an epimorphism $p$ co-recognizes morphisms if, whenever the composition of a map-transformation $v$ with $p$ results in a morphism $vp$, then, $v$ is a morphism itself.

In groups  (with zero-preserving maps) it is true that every monomorphism is a recognizer and every epimorphim is a co-recognizer. In monoids it is still true that every monomorphism is a recognizer but only surjective homomorphisms (that is, regular epimorphisms) are co-recognizers. In the category of topological groups (with zero-preserving maps, not necessarily continuous) it is not true in general that monomorphisms are recognizers. The same phenomenon can be observed in the case of preordered groups (with zero-preserving maps, not necessarily monotone). 

\begin{theorem}\label{Th: recognizers}
Let $(X,A,B,p,k,q,s)$ be a semi-biproduct in a category with map-transformations $\A=(\A,\map,\varepsilon,\mu)$, then:
\begin{enumerate}
\item the morphism $k$ is a monomorphism;
\item the morphism $p$ is an epimorphism;
\item if $k$ recognizes morphisms then it is the kernel of $p$;
\item if $p$ co-recognizes morphisms then it is the cokernel of $k$.
\end{enumerate}
\end{theorem}
\begin{proof}
Let $u,v\colon{Y\to X}$ be two morphisms in $\A$ with $ku=kv$. Then $qku=qkv$ and since $qk=\varepsilon(1_X)$ we have $\varepsilon(1_X)u=\varepsilon(1_X)v$. By the naturality of the inclusion of morphisms into map-transformations we have   $\varepsilon(1_X u)=\varepsilon(1_X v)$ which gives $\varepsilon(u)=\varepsilon(v)$. We may conclude that $u=v$ because $\varepsilon$ is an inclusion. This proves that $k$ is a monomorphism. Dually we prove that $p$ is an epimorphism.

Let $f\colon{Y\to A}$ be a morphism for which $pf=0$. Then there exists a unique morphism $f'\colon{Y\to X}$ such that $\varepsilon(f')=qf$. Indeed, the map-transformation $qf$ is such that $kqf=f$ (because $f=kqf+spf$ and $pf=0$). The existence of $f'$ comes from the fact that $k$ recognizes morphisms. This proves that $k$ is the kernel of $p$. Dually we prove that $p$ is the co-kernel of $k$.
\end{proof}

There is an analogue to the celebrated split short five lemma for semi-biproducts.

A morphism between semi-biproducts consists of three morphisms $(f_0,f_1,f_2)$ as illustrated
\[
\xymatrix{X\ar[d]_{f_0}\ar@<-.5ex>[r]_{k} & A\ar[d]^{f_1}\ar@<-.5ex>@{=>}[l]_{q}\ar@<.5ex>[r]^{p} & B\ar[d]^{f_2} \ar@{=>}@<.5ex>[l]^{s}\\
X'\ar@<-.5ex>[r]_{k'} & A'\ar@<-.5ex>@{=>}[l]_{q'}\ar@<.5ex>[r]^{p'} & B' \ar@{=>}@<.5ex>[l]^{s'}}
\]
such that $q'f_1=f_0q$, $p'f_1=f_2p$, $k'f_0=f_1k$ and $f_1s=s'f_2$.

\begin{theorem}
Let $(f_0,f_1,f_2)$ be a morphism of semi-biproducts as displayed above. If $f_0$ and $f_2$ are isomorphisms and $f_1$ recognizes or co-recognizes morphisms, then $f_1$ is an isomorphism whose inverse $f_1^{-1}$ is such that $\varepsilon(f_1^{-1})=kf_0^{-1}q'+sf_2^{-1}p'$.
\end{theorem}
\begin{proof}
Let us suppose, without loss of generality, that $X=X'$, $B=B'$, $f_0=1_X$, and $f_2=1_B$. Let us also denote $f_1$  simply by $f$. Then, it is clear that the map-transformation $g=kq'+sp'$ is such that $gf=(kq'+sp')f=kq'f+sp'f=kq+sp=\varepsilon(1_A)$. Similarly, $fg=\varepsilon(1_A')$. Hence, $f$ is both a monomorphism and an epimorphism. If $f$ recognizes or co-recognizes morphisms then it follows that $g$ is in fact a morphism and thus $f^{-1}=g$.
\end{proof}

\section{Conclusion}

We have tried to give some directions in which the study of semi-biproducts can be carried out in a general context. In seems appropriate to put it into the light of transformation structures and 2-cells. It would be interesting to explore this concept even further, moving beyond the pointed context. For any category $\A$, together with a 2-cell structure $(H,\dom,\cod,0,+)$, and assuming the existence of a natural transformation $\varepsilon\colon{\hom\to H}$, it would be interesting to see what is a sequence in $\A$ of the form
\[\xymatrix{X\ar[r]^{k}& A\ar[r]^{p} & B}\]
together with 2-cells $q\in H(A,X)$ and $s\in H(B,A)$ such that the following conditions hold
\begin{eqnarray*}
ps=\varepsilon(1_B)\quad &,&\quad
qk=\varepsilon(1_X)\\
\dom(kq)=\cod(sp)\quad &,&\quad
kq+sp=\varepsilon(1_A).
\end{eqnarray*}
In this case there are no zero-morphisms and the horizontal composition $\mu$ does not need to be present. This could be applied, for example, in the category of rings.

\section*{Acknowledgements}



This work is supported by the Fundação para a Ciência e a Tecnologia (FCT) and Centro2020 through the Project references: UID/Multi /04044/2019; PAMI - ROTEIRO/0328/2013 (Nº 022158); Next.parts (17963); and also by CDRSP and ESTG from the Polytechnic Institute of Leiria.



\end{document}